\documentclass[11pt,a4paper]{article}

\usepackage[a4paper,margin=2.5cm]{geometry}
\usepackage{amsmath,amssymb,amsfonts,amsthm,mathtools}
\usepackage{enumerate}
\usepackage{mathrsfs}
\usepackage[colorlinks=true]{hyperref}
\usepackage{cite}

\numberwithin{equation}{section}
\allowdisplaybreaks

\theoremstyle{plain}
\newtheorem{theorem}{Theorem}[section]
\newtheorem{lemma}[theorem]{Lemma}
\newtheorem{corollary}[theorem]{Corollary}
\newtheorem{proposition}[theorem]{Proposition}

\theoremstyle{definition}
\newtheorem{definition}[theorem]{Definition}

\theoremstyle{remark}
\newtheorem{remark}[theorem]{Remark}

\newcommand{\C}{\mathbb{C}}
\newcommand{\N}{\mathbb{N}}
\newcommand{\Nzero}{\mathbb{N}_{0}}
\newcommand{\R}{\mathbb{R}}
\newcommand{\Z}{\mathbb{Z}}
\newcommand{\Tq}{\mathbb{T}_{\theta}^{n}}
\newcommand{\Tn}{\mathbb{T}^{n}}
\newcommand{\A}{\mathcal{A}_{\theta}^{\infty}}
\newcommand{\trc}{\tau}

\newcommand{\Ical}{\mathcal{I}}
\newcommand{\norm}[1]{\left\lVert #1 \right\rVert}
\newcommand{\abs}[1]{\left\lvert #1 \right\rvert}

\newcommand{\Hesst}{\mathrm{Hess}_{\theta}}

\title{Higher-Order Heat Estimates and Semilinear Heat Equations on the Noncommutative Torus}

\author{Fulin Yang, Zhipeng Yang\thanks{Corresponding author: yangzhipeng326@163.com}}

\date{}

\AtEndDocument{%
  \par
  \bigskip
  \bigskip

  \noindent
  \textbf{Fulin Yang}\\[0.2em]
  \textsc{Beijing Institute of Mathematical Sciences and Applications, Beijing, China}\\[0.3em]
  \textsc{Yau Mathematical Sciences Center, Tsinghua University, Beijing, China}\\[0.3em]
  \textit{E-mail address}: \texttt{fulinyang@bimsa.cn}\\[1.5em]

  \noindent
  \textbf{Zhipeng Yang}\\[0.2em]
  \textsc{Department of Mathematics, Yunnan Normal University, Kunming, China}\\
  \textsc{Yunnan Key Laboratory of Modern Analytical Mathematics and Applications, Yunnan Normal University, Kunming, China}\\[0.3em]
  \textit{E-mail address}: \texttt{yangzhipeng326@163.com}%
}

\begin{document}

\maketitle

\begin{abstract}
We establish sharp higher-order heat estimates with complete bound on the noncommutative tori \(\Tq\) 
and show the optimality in the small-time order. 
As an application in polynomial semilinear heat equations on \(\Tq\),  we give local well-posedness, the
blow-up alternative, persistence of higher regularity, and instantaneous
smoothing in the Sobolev algebra scale \(H^k_\theta\), \(k>n/2\).
\end{abstract}

\textbf{Keywords:} Noncommutative torus; Fourier multipliers; semilinear heat equation.

\textbf{MSC2020:} 46L52, 42B15, 35K58.

\section{Introduction}

Heat semigroups are one of the basic tools in elliptic and parabolic analysis. In the classical setting, estimates for spatial and time derivatives of heat semigroups lead to regularization effects, Sobolev estimates, square-function identities, and well-posedness results for nonlinear evolution equations. On Riemannian manifolds, such estimates are closely connected with \(\Gamma\)-calculus, functional inequalities, stochastic derivative formulas, and heat-kernel asymptotics. We refer to the monograph of Bakry, Gentil, and Ledoux \cite{BGL14} for a systematic account. Gradient estimates for harmonic functions via heat semigroups were obtained by Thalmaier and Wang \cite{TW98}, and Bismut-type formulas for Hessians of heat semigroups were studied by Arnaudon, Plank, and Thalmaier \cite{APT03}.

The purpose of this paper is to develop a higher-order heat-semigroup framework on the noncommutative torus and to apply it to semilinear heat equations. The noncommutative torus is a central model space in noncommutative geometry. It was developed as a basic example of a noncommutative differentiable manifold by Connes \cite{Connes94} and Rieffel \cite{Rieffel90}. In this setting, the smooth algebra \(\A\) carries canonical derivations
\[
\delta_1,\ldots,\delta_n,
\]
which play the role of coordinate vector fields, and the associated nonnegative Laplacian is
\[
L=-\sum_{j=1}^{n}\delta_j^2.
\]
The heat semigroup \(P_t=e^{-tL}\) is diagonal with respect to the Fourier basis \(\{U^m:m\in\Z^n\}\). Therefore one may expect classical heat estimates on \(\mathbb T^n\) to have counterparts on \(\mathbb T_\theta^n\). For \(L^p\)-spaces over quantum tori, the natural formulation of these estimates uses completely bounded Fourier multipliers.

Harmonic analysis on quantum tori was developed in depth by Chen, Xu, and Yin \cite{ChenXuYin}. A key result for the present paper is their transference theorem, which identifies completely bounded \(L^p\) Fourier multipliers on quantum tori with the corresponding completely bounded multipliers on the classical torus for \(1<p<\infty\). Sobolev, Besov, and Triebel--Lizorkin spaces on quantum tori were later studied systematically by Xiong, Xu, and Yin \cite{XiongXuYin}; in particular, they established semigroup characterizations of these function spaces and multiplier results independent of the deformation parameter. Quantum differentiability on quantum tori was investigated by McDonald, Sukochev, and Xiong \cite{McDonaldSukochevXiong}. From the viewpoint of noncommutative diffusion geometry, Wirth and Zhang \cite{WirthZhang} obtained complete gradient estimates for quantum Markov semigroups, while Arhancet \cite{Arhancet} studied heat-kernel estimates and spectral dimension in a broader noncommutative setting.

Nonlinear partial differential equations in noncommutative settings have also attracted attention. Rosenberg \cite{Rosenberg08} studied noncommutative analogues of Laplace's equation and some nonlinear elliptic equations on noncommutative tori. More recently, McDonald \cite{McDonald24} developed paradifferential methods for nonlinear partial differential equations on noncommutative Euclidean spaces and applied them to nonlinear evolution equations. Recent works further show that Fourier multiplier methods on quantum tori are closely related to nonlinear equations. Sukochev, Tulenov, and Zanin \cite{STZ25} proved complete boundedness of Sobolev projections on quantum tori and gave applications. Shaimardan, Tastankul, and Tulenov \cite{STT26} obtained \(L^p\)-\(L^q\) Fourier multiplier estimates on noncommutative tori with applications to nonlinear equations. Ruzhansky, Shaimardan, and Tulenov \cite{RST26} proved H\"ormander-type multiplier theorems and Nikolskii inequalities on quantum tori, with embedding and heat-type consequences. Related \(L^p\)-\(L^q\) multiplier estimates on more general noncommutative spaces were obtained in \cite{RT25}. In another direction, Fathizadeh and Khalkhali \cite{FK12,FK13} studied Laplace-type operators and scalar curvature on noncommutative two-tori.

Motivated by these developments, we study pointwise-in-time estimates for the mixed heat operators
\[
\delta^\alpha L^\ell P_t,
\qquad \alpha\in\Nzero^n,
\qquad \ell\in\Nzero.
\]
Our first aim is to prove the optimal small-time estimate
\[
\|\delta^\alpha L^\ell P_t\|_{cb(L^p(\Tq))}
\lesssim t^{-\ell-\frac{|\alpha|}{2}},
\qquad 1<p<\infty,
\]
and to record its consequences in Sobolev regularization and Hessian-type estimates (see definition \ref{Def2.1} for the complete bound). Our second aim is to use this heat-semigroup framework to study semilinear heat equations on \(\Tq\). Since the product on \(\Tq\) is noncommutative, the nonlinear estimates required for polynomial nonlinearities are not merely pointwise estimates. We therefore prove, within the paper, a Sobolev algebra theorem and polynomial nonlinear estimates in the Hilbert scale \(H^k_\theta\), \(k>n/2\), using the twisted Fourier convolution formula.

We state the main results. For a multi-index \(\alpha=(\alpha_1,\ldots,\alpha_n)\in\Nzero^n\), set
\[
|\alpha|=\alpha_1+\cdots+\alpha_n,
\qquad
\delta^\alpha=\delta_1^{\alpha_1}\cdots\delta_n^{\alpha_n},
\qquad
m^\alpha=m_1^{\alpha_1}\cdots m_n^{\alpha_n}.
\]

\begin{theorem}\label{Thm1.1}
Let \(\alpha\in\Nzero^{n}\), let \(\ell\in\Nzero\), and let \(1<p<\infty\). 
\begin{enumerate}[\rm(i)]
\item For every \(t>0\), the operator \(\delta^{\alpha}L^{\ell}P_{t}\), defined on \(\A\), extends uniquely to a completely bounded operator on \(L^{p}(\Tq)\). Moreover, there exists a constant \(C_{\alpha,\ell,n}>0\), independent of \(\theta\) and \(p\), such that
\[
\norm{\delta^{\alpha}L^{\ell}P_{t}}_{cb(L^{p}(\Tq))}
\le
\frac{C_{\alpha,\ell,n}}{t^{\ell+\abs{\alpha}/2}}
\qquad \forall\; t>0.
\]
In particular, for every \(a\in\A\),
\[
\norm{\delta^{\alpha}\partial_{t}^{\ell}P_{t}(a)}_{L^{p}(\Tq)}
\le
\frac{C_{\alpha,\ell,n}}{t^{\ell+\abs{\alpha}/2}}
\norm{a}_{L^{p}(\Tq)}
\qquad \forall\; t>0.
\]
\item If \((\alpha,\ell)\neq(0,0)\), then there exists a constant \(c_{\alpha,\ell,n}>0\) such that, for every \(1<p<\infty\) and all sufficiently small \(t>0\),
\[
\sup_{\norm{a}_{L^{p}(\Tq)}=1}
\norm{\delta^{\alpha}L^{\ell}P_{t}(a)}_{L^{p}(\Tq)}
\ge
\frac{c_{\alpha,\ell,n}}{t^{\ell+\abs{\alpha}/2}}.
\]
Hence the factor \(t^{-\ell-\abs{\alpha}/2}\) in {\rm (i)} cannot be improved as \(t\downarrow0\).
\end{enumerate}
\end{theorem}

The Theorem~\ref{Thm1.1} implies Sobolev regularization. For every \(k,r\in\Nzero\) and \(1<p<\infty\),
\[
P_t:W^{k,p}_\theta\longrightarrow W^{k+r,p}_\theta
\]
and
\[
\|P_t a\|_{W^{k+r,p}_\theta}
\le
C_{k,r,n}\bigl(1+t^{-r/2}\bigr)\|a\|_{W^{k,p}_\theta}.
\]
Thus \(P_t\) is instantly smoothing in every Sobolev scale generated by the canonical derivations.

The nonlinear estimates behind the semilinear theory are as follows. If \(k>n/2\), then
\[
\|ab\|_{H^k_\theta}
\le
C_{k,n}\|a\|_{H^k_\theta}\|b\|_{H^k_\theta},
\qquad a,b\in H^k_\theta.
\]
Consequently, every noncommutative polynomial of the form
\[
\mathcal P(u)=
\sum_{\nu=0}^{q}\sum_{\mu=1}^{N_\nu}
 b_{\nu,\mu,0}u b_{\nu,\mu,1}u\cdots u b_{\nu,\mu,\nu},
\qquad b_{\nu,\mu,j}\in H^k_\theta,
\]
defines a locally Lipschitz continuous map $\mathcal P:H^k_\theta\to H^k_\theta$, i.e., for every \(R>0\), there exists \(C_R>0\) such that
\[
\|\mathcal P(u)-\mathcal P(v)\|_{H^k_\theta}
\le C_R\|u-v\|_{H^k_\theta}
\]
whenever \(\|u\|_{H^k_\theta},\|v\|_{H^k_\theta}\le R\).

\begin{theorem}\label{Thm1.2}
Let \(k\in\N\) satisfy \(k>n/2\), and assume that \(\mathcal N:H^{k}_{\theta}\to H^{k}_{\theta}\) is locally Lipschitz.
\begin{enumerate}[\rm(i)]
\item For every \(u_0\in H^{k}_{\theta}\), there exist \(T>0\) and a unique mild solution
$u\in C([0,T];H^{k}_{\theta})$
of
\[
\partial_t u+Lu=\mathcal N(u),
\qquad u(0)=u_0.
\]
\item The solution has a maximal existence time \(T_{\max}\in(0,\infty]\). If \(T_{\max}<\infty\), then
\[
\lim_{t\to T_{\max}}\|u(t)\|_{H^k_\theta}=\infty.
\]
\item The solution depends locally Lipschitz continuously on the initial datum, i.e., if two solutions \(u\) and \(v\) are defined on \([0,T]\) and stay in a fixed ball of \(H^k_\theta\), then
\[
\sup_{0\le t\le T}\|u(t)-v(t)\|_{H^k_\theta}
\le C_T\|u(0)-v(0)\|_{H^k_\theta}.
\]
\item In addition, for every \(0<T<T_{\max}\), one has
$u\in C((0,T];H^{k+1}_\theta)$, 
and, for every \(0<t\le T\),
\begin{equation}\label{eq1.1}
\|u(t)\|_{H^{k+1}_\theta}
\le
C_{k,n}(1+t^{-1/2})\|u_0\|_{H^k_\theta}
+C_{k,n}M_T(t+t^{1/2}),
\end{equation}
where
\[
M_T=\sup_{0\le s\le T}\|\mathcal N(u(s))\|_{H^k_\theta}.
\]
\item If \(r\in\N\), \(u_0\in H^{k+r}_\theta\), and \(\mathcal N\) is locally Lipschitz from \(H^{k+r}_\theta\) into itself, then
\[
u\in C([0,T];H^{k+r}_\theta)
\]
for every \(0<T<T_{\max}\).
\end{enumerate}
\end{theorem}

As a direct consequence, polynomial semilinear heat equations are locally well posed in the algebra scale \(H^k_\theta\), \(k>n/2\) (where "locally" corresponds to "local time"). If the polynomial coefficients belong to \(H^{k+r}_\theta\) for all \(r\in\N\), then the corresponding solution is instantly smooth:
\[
u\in C((0,T];H^{k+r}_\theta)
\qquad \forall r\in\N,
\qquad 0<T<T_{\max}.
\]

The paper is organized as follows. Section~\ref{sec:preliminaries} recalls the noncommutative torus, the canonical derivations, Sobolev norms, and the heat semigroup. Section~\ref{sec:main-estimates} proves the Fourier identities, the completely bounded heat estimates, Theorem~\ref{Thm1.1}, and the Sobolev regularization consequences. Section~\ref{sec:sobolev-algebra} proves the Sobolev algebra property and polynomial nonlinear estimates in \(H^k_\theta\). Section~\ref{sec:semilinear-heat} proves Theorem~\ref{Thm1.2} and the polynomial consequences. 

\section{Preliminaries}\label{sec:preliminaries}

We use standard notation for noncommutative tori; see, for example, \cite{Rieffel90,ChenXuYin,XiongXuYin,McDonaldSukochevXiong}. Let \(n\ge 1\), and let $\theta=(\theta_{jk})_{1\le j,k\le n}$ be a real skew-symmetric \(n\times n\) matrix. Given a Hilbert space $\mathcal{H}$, let $U_{1},\dots,U_{n}$ be the unitary operators on $\mathcal{H}$ satisfying
\[
U_{k}U_{j}=e^{2\pi i\theta_{kj}}U_{j}U_{k},
\qquad 1\le j,k\le n.
\]
The noncommutative torus $\mathbb{T}_{\theta}^{n}$ is the von Neumann algebra generated by these unitaries, also denoted by $L^{\infty}(\mathbb{T}_{\theta}^{n})$. 

Denote by \(\mathcal{A}_{\theta}\) the \(C^{*}\)-algebra generated by $U_{1},\dots,U_{n}$.
For \(m=(m_{1},\dots,m_{n})\in\Z^{n}\), we write
\[
U^{m}=U_{1}^{m_{1}}\cdots U_{n}^{m_{n}}.
\]
The smooth algebra \(\A\) consists of all series
\[
a=\sum_{m\in\Z^{n}} c_{m} U^{m}
\]
with rapidly decreasing Fourier coefficients \((c_{m})_{m\in\Z^{n}}\). The canonical trace is given on \(\A\) by
\[\trc(a)=c_{(0,\ldots,0)}.\]
For \(1\le p<\infty\), the noncommutative \(L^{p}\)-space \(L^{p}(\Tq)\) is the completion of \(\A\) with respect to the norm
\[
\|a\|_{L^{p}(\Tq)}=\bigl(\trc(|a|^{p})\bigr)^{\frac{1}{p}}.
\]
For \(p=\infty\), we put \(L^{\infty}(\mathbb T_{\theta}^{n})=\Tq\).

\begin{definition}\label{Def2.1}
Let \(1\le p<\infty\). For \(r\in\mathbb N\), we identify
\[
M_{r}(L^{p}(\Tq))
=
L^{p}\bigl(L^{\infty}(\mathbb T_{\theta}^{n})\bar\otimes M_{r},
\tau\otimes \operatorname{tr}_{r}\bigr),
\]
where \(\operatorname{tr}_{r}\) is the usual trace on \(M_{r}\). If
\(T:L^{p}(\Tq)\to L^{p}(\Tq)\) is linear, we define
\[
\|T\|_{cb(L^{p}(\Tq))}
=
\sup_{r\ge 1}
\|\operatorname{id}_{M_{r}}\otimes T\|_{M_{r}(L^{p}(\Tq))\to M_{r}(L^{p}(\Tq))}.
\]
We say that \(T\) is completely bounded if this quantity is finite.
\end{definition}

\begin{remark}
The same convention is used for the classical torus. Thus the completely bounded norm of a Fourier multiplier on \(L^{p}(\Tn)\) is computed after amplification to
\(L^{p}(L^{\infty}(\Tn)\bar\otimes M_{r})\), equivalently to \(L^{p}(\Tn;S_{r}^{p})\) with the Schatten \(p\)-norm in the matrix variable. More details, we refer to \cite{P1988}.
\end{remark}

For \(a\in\A\), the Fourier coefficients are defined by
\[
\widehat a(m)=\trc\bigl((U^{m})^{*}a\bigr),
\qquad m\in\Z^{n}.
\]
Clearly, the family \(\{U^{m}:m\in\Z^{n}\}\) is an orthonormal basis of \(L^{2}(\Tq)\) and every \(a\in L^{2}(\Tq)\) admits an \(L^{2}\)-Fourier expansion
\[
a=\sum_{m\in\Z^{n}}\widehat a(m)U^{m}.
\]
Moreover, Plancherel's identity holds
\[
\norm{a}_{L^{2}(\Tq)}^{2}=\sum_{m\in\Z^{n}}\abs{\widehat a(m)}^{2}.
\]

For each \(1\le j\le n\), the canonical derivation \(\delta_{j}\) on $\Tq$ is defined by
\[
\delta_{j}(U^{m})=2\pi i\,m_{j}U^{m},
\qquad m\in\Z^{n}.
\]
If \(\alpha=(\alpha_{1},\dots,\alpha_{n})\in\Nzero^{n}\), set
\[
\abs{\alpha}=\alpha_{1}+\cdots+\alpha_{n},
\qquad
\delta^{\alpha}=\delta_{1}^{\alpha_{1}}\cdots\delta_{n}^{\alpha_{n}},
\qquad
m^{\alpha}=m_{1}^{\alpha_{1}}\cdots m_{n}^{\alpha_{n}}.
\]
Since the canonical derivations commute on the Fourier basis, \(\delta^{\alpha}\) is well defined on \(\A\).
For $a\in\A$ of the form $\sum_{m\in\Z^{n}}c_{m}U^{m}$,
one has
\[
\delta^{\alpha}(a)=(2\pi i)^{|\alpha|}\sum_{m\in\Z^{n}}m^{\alpha}c_{m}U^{m}.
\]

The  Laplacian is defined by
\[
L=-\sum_{j=1}^{n}\delta_{j}^{2}.
\]
On the Fourier basis,
\[
L(U^{m})=4\pi^{2}\abs{m}^{2}U^{m},
\qquad\mbox{where}\quad
\abs{m}^{2}=m_{1}^{2}+\cdots+m_{n}^{2}.
\]
For \(k\in\Nzero\), let
\[
\Ical_{k}=\{\alpha\in\Nzero^{n}: \abs{\alpha}=k\}
\quad\mbox{and}\quad
N_{k,n}=\#\Ical_{k}=\binom{n+k-1}{k}.
\]

\begin{definition}
For \(a\in\A\), define the \(k\)-th order differential family and the  Hessian by
\[
\nabla_{\theta}^{k}(a)=\bigl(\delta^{\alpha}a\bigr)_{\alpha\in\Ical_{k}},
\qquad
\Hesst(a)=\bigl(\delta_{i}\delta_{j}a\bigr)_{1\le i,j\le n}.
\]
\end{definition}

\begin{definition}
Let \(1\le p\le\infty\). For a family \(Y=(y_{\alpha})_{\alpha\in\Ical_{k}}\) with entries in \(L^{p}(\Tq)\), we define the entrywise \(\ell^{2}\)-norm by
\[
\norm{Y}_{\ell^{2}(\Ical_{k};L^{p}(\Tq))}
=
\left(\sum_{\alpha\in\Ical_{k}}\norm{y_{\alpha}}_{L^{p}(\Tq)}^{2}\right)^{1/2}.
\]
For a matrix \(X=(x_{ij})_{1\le i,j\le n}\) with entries in \(L^{p}(\Tq)\), set
\[
\norm{X}_{\ell^{2}_{n}(L^{p}(\Tq))}
=
\left(\sum_{i,j=1}^{n}\norm{x_{ij}}_{L^{p}(\Tq)}^{2}\right)^{1/2}.
\]
\end{definition}

\begin{definition}
Let \(1\le p<\infty\) and \(k\in\Nzero\). For \(a\in\A\), define the Sobolev seminorm and norm
\[
\abs{a}_{W^{k,p}_{\theta}}=\norm{\nabla_{\theta}^{k}(a)}_{\ell^{2}(\Ical_{k};L^{p}(\Tq))},
\qquad
\norm{a}_{W^{k,p}_{\theta}}=\sum_{j=0}^{k}\abs{a}_{W^{j,p}_{\theta}},
\]
where \(\Ical_{0}=\{0\}\) and \(\delta^{0}=\mathrm{Id}\). We denote by \(W^{k,p}_{\theta}\) the completion of \(\A\) with respect to \(\norm{\cdot}_{W^{k,p}_{\theta}}\).
\end{definition}

\begin{definition}\label{Def2.6}
For \(s\ge0\), define the Hilbert Sobolev space \(H^{s}_{\theta}\) as the set of all \(a\in L^{2}(\Tq)\) such that
\[
\norm{a}_{H^{s}_{\theta}}^{2}
=
\sum_{m\in\Z^{n}}\bigl(1+4\pi^{2}\abs{m}^{2}\bigr)^{s}\abs{\widehat a(m)}^{2}<\infty .
\]
\end{definition}

\begin{lemma}\label{Lem2.7}
For every integer \(k\ge0\), the norm \(\norm{\cdot}_{H^{k}_{\theta}}\) is equivalent to \(\norm{\cdot}_{W^{k,2}_{\theta}}\). In particular, \(H^{k}_{\theta}\) and \(W^{k,2}_{\theta}\) have the same elements.
\end{lemma}

\begin{proof}
It suffices to prove the equivalence on \(\A\), and then pass to completions. For \(a\in\A\), Plancherel's identity gives
\[
\norm{a}_{W^{k,2}_{\theta}}^{2}
\simeq
\sum_{\abs{\alpha}\le k}\norm{\delta^{\alpha}a}_{L^{2}(\Tq)}^{2}
=
\sum_{m\in\Z^{n}}\left(\sum_{\abs{\alpha}\le k}(2\pi)^{2\abs{\alpha}}\abs{m^{\alpha}}^{2}\right)\abs{\widehat a(m)}^{2}.
\]
The polynomial in parentheses is comparable to \(\bigl(1+4\pi^{2}\abs{m}^{2}\bigr)^{k}\), with constants depending only on \(k\) and \(n\). This proves the equivalence.
\end{proof}

\begin{remark}
By finite-dimensional norm equivalence on each \(\Ical_{j}\), the norm \(\norm{\cdot}_{W^{k,p}_{\theta}}\) is equivalent to the more standard derivative-based Sobolev norm
\[
\sum_{0\le \abs{\alpha}\le k}\norm{\delta^{\alpha}a}_{L^{p}(\Tq)},
\]
so the completion above coincides with the usual Sobolev scale considered in \cite{XiongXuYin}.
\end{remark}

\begin{remark}
The objects \(\nabla_{\theta}^{k}\) and \(\Hesst\) depend on the distinguished  derivations \(\delta_{1},\dots,\delta_{n}\). Throughout the paper they are used as flat analogues of classical higher-order differential operators rather than as coordinate-free tensors.
\end{remark}

The heat semigroup generated by \(L\) is
\[
P_{t}=e^{-tL},
\qquad t>0.
\]
Since \(L\) is diagonal on Fourier modes,
\[
P_{t}(U^{m})=e^{-4\pi^{2}\abs{m}^{2}t}U^{m},
\qquad m\in\Z^{n}.
\]
Hence, for $\displaystyle  a=\sum_{m\in\Z^{n}}\widehat a(m)U^{m}\in\A$,
one has
\[
P_{t}(a)=\sum_{m\in\Z^{n}}e^{-4\pi^{2}\abs{m}^{2}t}\widehat a(m)U^{m}.
\]

\section{Main estimates and Sobolev regularization}\label{sec:main-estimates}

We first record the Fourier identities used below.

\begin{lemma}\label{Lem3.1}
Let \(\alpha\in\Nzero^{n}\), \(\ell\in\Nzero\), and let $\displaystyle a=\sum_{m\in\Z^{n}}\widehat a(m)U^{m}\in\A.$
Then, for every \(t>0\),
\[
\delta^{\alpha}P_{t}(a)=P_{t}(\delta^{\alpha}a)
\]
and
\[
\delta^{\alpha}P_{t}(a)
=
(2\pi i)^{\abs{\alpha}}
\sum_{m\in\Z^{n}}
m^{\alpha}e^{-4\pi^{2}\abs{m}^{2}t}\widehat a(m)U^{m}.
\]
Moreover,
\[
\delta^{\alpha}L^{\ell}P_{t}(a)
=
(2\pi i)^{\abs{\alpha}}
\sum_{m\in\Z^{n}}(4\pi^{2}\abs{m}^{2})^{\ell}m^{\alpha}
e^{-4\pi^{2}\abs{m}^{2}t}\widehat a(m)U^{m},
\]
and
\[
\partial_{t}^{\ell}P_{t}(a)=(-L)^{\ell}P_{t}(a).
\]
\end{lemma}

\begin{proof}
All identities follow by termwise application to the rapidly convergent Fourier series. Indeed,
\[
P_{t}(U^{m})=e^{-4\pi^{2}\abs{m}^{2}t}U^{m},
\qquad
\delta^{\alpha}U^{m}=(2\pi i)^{\abs{\alpha}}m^{\alpha}U^{m},
\]
and
\[
L(U^{m})=4\pi^{2}\abs{m}^{2}U^{m}.
\]
These formulas give the commutation relation, the two multiplier representations, and the identity for time derivatives.
\end{proof}

We next pass through the classical torus \(\Tn=\R^{n}/\Z^{n}\).

\begin{lemma}\label{Lem3.2}
Let \(\alpha\in\Nzero^{n}\), and let
\[
G_{t}(x)=(4\pi t)^{-n/2}e^{-\abs{x}^{2}/(4t)},
\qquad x\in\R^{n},\quad t>0,
\]
be the Euclidean heat kernel. Then there exists a constant \(C_{\alpha,n}>0\) such that
\[
\norm{\partial^{\alpha}G_{t}}_{L^{1}(\R^{n})}
\le
\frac{C_{\alpha,n}}{t^{\abs{\alpha}/2}},
\qquad \forall\; t>0.
\]
\end{lemma}

\begin{proof}
The scaling identity
\[
G_{t}(x)=t^{-n/2}G_{1}(x/\sqrt{t})
\]
implies
\[
\partial^{\alpha}G_{t}(x)
=
t^{-n/2-\abs{\alpha}/2}(\partial^{\alpha}G_{1})(x/\sqrt{t}).
\]
Hence, by the change of variables \(x=\sqrt{t}\,y\),
\[
\norm{\partial^{\alpha}G_{t}}_{L^{1}(\R^{n})}
=
t^{-\abs{\alpha}/2}\int_{\R^{n}}\abs{(\partial^{\alpha}G_{1})(y)}\,dy.
\]
Since \(G_{1}\) is a Schwartz function, \(\partial^{\alpha}G_{1}\in L^{1}(\R^{n})\), and the conclusion follows.
\end{proof}

\begin{lemma}\label{Lem3.3}
Let \(\alpha\in\Nzero^{n}\), and let
\[
H_{t}(x)=\sum_{k\in\Z^{n}}G_{t}(x+k),
\qquad x\in\Tn,
\]
be the heat kernel on the classical torus \(\Tn\). Then there exists a constant \(C_{\alpha,n}>0\) such that
\[
\norm{\partial^{\alpha}H_{t}}_{L^{1}(\Tn)}
\le
\frac{C_{\alpha,n}}{t^{\abs{\alpha}/2}},
\qquad \forall\; t>0.
\]
\end{lemma}

\begin{proof}
Since derivatives of the Gaussian are rapidly decreasing, the periodized series defining \(\partial^{\alpha}H_{t}\) converges absolutely in \(L^{1}(\Tn)\), and
\[
\partial^{\alpha}H_{t}(x)=\sum_{k\in\Z^{n}}\partial^{\alpha}G_{t}(x+k).
\]
Therefore,
\[
\begin{aligned}
\norm{\partial^{\alpha}H_{t}}_{L^{1}(\Tn)}
\le \sum_{k\in\Z^{n}}\int_{[0,1]^{n}}\abs{\partial^{\alpha}G_{t}(x+k)}\,dx 
= \int_{\R^{n}}\abs{\partial^{\alpha}G_{t}(x)}\,dx
= \norm{\partial^{\alpha}G_{t}}_{L^{1}(\R^{n})}.
\end{aligned}
\]
The desired estimate follows from Lemma~\ref{Lem3.2}.
\end{proof}

\begin{lemma}\label{Lem3.4}
Let \(\alpha\in\Nzero^{n}\), and for \(t>0\) define the symbol
\[
M_{t}^{\alpha}(m)=(2\pi i)^{\abs{\alpha}}m^{\alpha}e^{-4\pi^{2}\abs{m}^{2}t},
\qquad m\in\Z^{n}.
\]
Let \(T_{t}^{\alpha}\) be the corresponding Fourier multiplier on \(\Tn\). Then, for every \(1\le p<\infty\),
\[
\norm{T_{t}^{\alpha}}_{cb(L^{p}(\Tn))}
\le
\frac{C_{\alpha,n}}{t^{\abs{\alpha}/2}},
\qquad \forall\; t>0.
\]
\end{lemma}

\begin{proof}
The Fourier series of \(H_{t}\) is
\[
H_{t}(x)=\sum_{m\in\Z^{n}}e^{-4\pi^{2}\abs{m}^{2}t}e^{2\pi i m\cdot x},
\]
so
\[
\partial^{\alpha}H_{t}(x)
=
\sum_{m\in\Z^{n}}M_{t}^{\alpha}(m)e^{2\pi i m\cdot x}.
\]
Thus \(T_{t}^{\alpha}\) is convolution by the scalar kernel \(\partial^{\alpha}H_{t}\).

Let \(r\in\N\). For \(F\in L^{p}(\Tn;S_{r}^{p})\), the amplification \(\operatorname{id}_{M_{r}}\otimes T_{t}^{\alpha}\) acts by convolution with the scalar kernel:
\[
(\operatorname{id}_{M_{r}}\otimes T_{t}^{\alpha})F=(\partial^{\alpha}H_{t})*F.
\]
By the vector-valued Young inequality for convolution on \(\Tn\), applied to \(S_{r}^{p}\)-valued functions,
\[
\norm{(\operatorname{id}_{M_{r}}\otimes T_{t}^{\alpha})F}_{L^{p}(\Tn;S_{r}^{p})}
\le
\norm{\partial^{\alpha}H_{t}}_{L^{1}(\Tn)}
\norm{F}_{L^{p}(\Tn;S_{r}^{p})}.
\]
Using Lemma~\ref{Lem3.3}, we obtain
\[
\norm{(\operatorname{id}_{M_{r}}\otimes T_{t}^{\alpha})F}_{L^{p}(\Tn;S_{r}^{p})}
\le
\frac{C_{\alpha,n}}{t^{\abs{\alpha}/2}}
\norm{F}_{L^{p}(\Tn;S_{r}^{p})}.
\]
The constant is independent of \(r\). Taking the supremum over all \(r\in\N\) gives the completely bounded estimate.
\end{proof}

\begin{lemma}\label{Lem3.5}
Let \(1<p<\infty\), and let \(\varphi:\Z^{n}\to\C\). Denote by \(M^{0}_{\varphi}\) the Fourier multiplier associated with \(\varphi\) on the classical torus \(\Tn\). Assume that \(M^{0}_{\varphi}\) is completely bounded on \(L^{p}(\Tn)\). Then the Fourier multiplier \(M^{\theta}_{\varphi}\) associated with the same symbol on \(\Tq\) is completely bounded on \(L^{p}(\Tq)\), and
\[
\norm{M^{\theta}_{\varphi}}_{cb(L^{p}(\Tq))}
=
\norm{M^{0}_{\varphi}}_{cb(L^{p}(\Tn))}.
\]
\end{lemma}

\begin{proof}
This is the transference theorem for completely bounded Fourier multipliers on quantum tori, see \cite[Theorem~7.3]{ChenXuYin}.
\end{proof}

\begin{proposition}\label{Prop3.6}
Let \(\alpha\in\Nzero^{n}\) and \(1<p<\infty\). For every \(t>0\), the operator \(\delta^{\alpha}P_{t}\), initially defined on \(\A\), extends uniquely to a completely bounded operator on \(L^{p}(\Tq)\). Moreover, there exists a constant \(C_{\alpha,n}>0\), independent of \(\theta\) and \(p\), such that
\[
\norm{\delta^{\alpha}P_{t}}_{cb(L^{p}(\Tq))}
\le
\frac{C_{\alpha,n}}{t^{\abs{\alpha}/2}},
\qquad \forall\; t>0.
\]
Consequently,
\[
\norm{\delta^{\alpha}P_{t}(a)}_{L^{p}(\Tq)}
\le
\frac{C_{\alpha,n}}{t^{\abs{\alpha}/2}}
\norm{a}_{L^{p}(\Tq)},
\qquad \forall\; a\in L^{p}(\Tq),\quad t>0.
\]
\end{proposition}

\begin{proof}
By Lemma \ref{Lem3.1}, \(\delta^{\alpha}P_{t}\) is, on \(\A\), the Fourier multiplier with symbol
\[
M_{t}^{\alpha}(m)=(2\pi i)^{\abs{\alpha}}m^{\alpha}e^{-4\pi^{2}\abs{m}^{2}t},
\qquad m\in\Z^{n}.
\]
Let \(T_{t}^{\alpha}\) be the Fourier multiplier on the classical torus \(\Tn\) with the above symbol. By Lemma~\ref{Lem3.4},
\[
\norm{T_{t}^{\alpha}}_{cb(L^{p}(\Tn))}
\le
\frac{C_{\alpha,n}}{t^{\abs{\alpha}/2}}.
\]
Applying Lemma~\ref{Lem3.5} to the symbol \(M_{t}^{\alpha}\), we conclude that \(\delta^{\alpha}P_{t}\) extends to a completely bounded Fourier multiplier on \(L^{p}(\Tq)\) satisfying the same estimate.
\end{proof}

We derive the main consequences of the basic \(L^{p}\)-estimate.

\begin{proof}[Proof of Theorem \ref{Thm1.1}]
We first prove (i). Since the derivations commute, write
\[
L^{\ell}=(-1)^{\ell}\Bigl(\sum_{j=1}^{n}\delta_{j}^{2}\Bigr)^{\ell}
=
(-1)^{\ell}\sum_{\beta\in\Ical_{\ell}}\binom{\ell}{\beta}\delta^{2\beta},
\]
where
\[
\binom{\ell}{\beta}=\frac{\ell!}{\beta_{1}!\cdots\beta_{n}!},
\qquad
2\beta=(2\beta_{1},\dots,2\beta_{n}).
\]
Therefore, on \( \A \),
\[
\delta^{\alpha}L^{\ell}P_{t}
=
(-1)^{\ell}\sum_{\beta\in\Ical_{\ell}}\binom{\ell}{\beta}\delta^{\alpha+2\beta}P_{t}.
\]
Since \(\abs{\alpha+2\beta}=\abs{\alpha}+2\ell\), Proposition~\ref{Prop3.6} gives
\[
\begin{aligned}
\norm{\delta^{\alpha}L^{\ell}P_{t}}_{cb(L^{p}(\Tq))}
&\le \sum_{\beta\in\Ical_{\ell}}\binom{\ell}{\beta}\norm{\delta^{\alpha+2\beta}P_{t}}_{cb(L^{p}(\Tq))} \\
&\le \sum_{\beta\in\Ical_{\ell}}\binom{\ell}{\beta}\frac{C_{\alpha+2\beta,n}}{t^{\ell+\abs{\alpha}/2}} \\
&\le \frac{C_{\alpha,\ell,n}}{t^{\ell+\abs{\alpha}/2}}.
\end{aligned}
\]
The density of \( \A \) in \(L^{p}(\Tq)\) gives the required extension. Moreover, Lemma~\ref{Lem3.1} gives
\[
\partial_{t}^{\ell}P_{t}(a)=(-L)^{\ell}P_{t}(a),
\qquad a\in\A,
\]
and the time-derivative estimate follows from the preceding bound.

It remains to prove (ii). Assume \((\alpha,\ell)\neq (0,0)\). For \(t>0\), set
\[
k_{t}=\left\lfloor \frac{1}{\sqrt{8\pi^{2}nt}}\right\rfloor,
\qquad
m_{t}=(k_{t},\dots,k_{t})\in\Z^{n},
\qquad
a_{t}=U^{m_{t}}.
\]
For all sufficiently small \(t>0\), one has \(k_t\ge1\), and \(U^{m_t}\) is unitary. Hence
\[
\norm{a_{t}}_{L^{p}(\Tq)}=1,
\qquad 1\le p\le\infty.
\]
Using Lemma~\ref{Lem3.1},
\[
\delta^{\alpha}L^{\ell}P_{t}(a_{t})
=
(2\pi i)^{\abs{\alpha}}(4\pi^{2}\abs{m_{t}}^{2})^{\ell}m_{t}^{\alpha}e^{-4\pi^{2}\abs{m_{t}}^{2}t}U^{m_{t}}.
\]
Therefore
\[
\norm{\delta^{\alpha}L^{\ell}P_{t}(a_{t})}_{L^{p}(\Tq)}
=
(2\pi)^{\abs{\alpha}}(4\pi^{2}\abs{m_{t}}^{2})^{\ell}\abs{m_{t}^{\alpha}}e^{-4\pi^{2}\abs{m_{t}}^{2}t}.
\]
Since all coordinates of \(m_t\) are equal,
\[
\abs{m_{t}^{\alpha}}=k_{t}^{\abs{\alpha}},
\qquad
\abs{m_{t}}^{2}=nk_{t}^{2}.
\]
For all sufficiently small \(t>0\),
\[
k_{t}\ge \frac{1}{2\sqrt{8\pi^{2}nt}},
\]
and hence
\[
k_{t}^{\abs{\alpha}}\ge c_{\alpha,n}t^{-\abs{\alpha}/2},
\qquad
(nk_{t}^{2})^{\ell}\ge c_{\ell,n}t^{-\ell}.
\]
Moreover,
\[
4\pi^{2}\abs{m_{t}}^{2}t
=
4\pi^{2}nk_{t}^{2}t
\le
\frac{1}{2},
\]
so
\[
e^{-4\pi^{2}\abs{m_{t}}^{2}t}\ge e^{-1/2}.
\]
Combining these estimates, we obtain
\[
\norm{\delta^{\alpha}L^{\ell}P_{t}(a_{t})}_{L^{p}(\Tq)}
\ge
\frac{c_{\alpha,\ell,n}}{t^{\ell+\abs{\alpha}/2}}
\]
for all sufficiently small \(t>0\). Since the completely bounded norm dominates the usual \(L^p\)-operator norm, the same lower bound shows that the completely bounded estimate in (i) cannot have a better small-time power.
\end{proof}

\begin{corollary}\label{Cor3.7}
Let \(k,r\in\Nzero\) and \(1<p<\infty\). Then there exists a constant \(C_{k,r,n}>0\), independent of \(\theta\) and \(p\), such that for every \(a\in\A\) and every \(t>0\),
\[
\abs{P_{t}a}_{W^{k+r,p}_{\theta}}
\le
\frac{C_{k,r,n}}{t^{r/2}}\abs{a}_{W^{k,p}_{\theta}}.
\]
Moreover,
\[
\norm{P_{t}a}_{W^{k+r,p}_{\theta}}
\le
C_{k,r,n}\bigl(1+t^{-r/2}\bigr)\norm{a}_{W^{k,p}_{\theta}}.
\]
Consequently, \(P_{t}\) extends to a bounded map from \(W^{k,p}_{\theta}\) into \(W^{k+r,p}_{\theta}\) for every \(t>0\).
\end{corollary}

\begin{proof}
Let
\[
\widetilde C_{r,n}=\max_{0\le s\le r}\ \max_{\abs{\eta}=s} C_{\eta,n}.
\]
For each \(\beta\in\Ical_{k+r}\), choose a multi-index \(\gamma(\beta)\in\Nzero^{n}\) such that
\[
0\le \gamma(\beta)\le \beta,
\qquad
\abs{\gamma(\beta)}=k,
\]
and set \(\eta(\beta)=\beta-\gamma(\beta)\), so that \(\abs{\eta(\beta)}=r\). Such a choice is possible because \(\abs{\beta}=k+r\). Using Lemma \ref{Lem3.1} and Proposition \ref{Prop3.6}, we obtain
\[
\norm{\delta^{\beta}P_{t}(a)}_{L^{p}(\Tq)}
=
\norm{\delta^{\eta(\beta)}P_{t}(\delta^{\gamma(\beta)}a)}_{L^{p}(\Tq)}
\le
\frac{\widetilde C_{r,n}}{t^{r/2}}\norm{\delta^{\gamma(\beta)}a}_{L^{p}(\Tq)}.
\]
Therefore,
\[
\begin{aligned}
\abs{P_{t}a}_{W^{k+r,p}_{\theta}}^{2}
&= \sum_{\beta\in\Ical_{k+r}}\norm{\delta^{\beta}P_{t}(a)}_{L^{p}(\Tq)}^{2} \\
&\le \frac{\widetilde C_{r,n}^{2}}{t^{r}}\sum_{\beta\in\Ical_{k+r}}\norm{\delta^{\gamma(\beta)}a}_{L^{p}(\Tq)}^{2} \\
&\le \frac{C_{k,r,n}^{2}}{t^{r}}\sum_{\gamma\in\Ical_{k}}\norm{\delta^{\gamma}a}_{L^{p}(\Tq)}^{2}\\
&= \frac{C_{k,r,n}^{2}}{t^{r}}\abs{a}_{W^{k,p}_{\theta}}^{2},
\end{aligned}
\]
which proves the seminorm estimate.

For the full Sobolev norm, let \(0\le j\le k+r\) and \(\beta\in\Ical_{j}\). Put \(s=\max\{0,j-k\}\), so that \(0\le s\le r\). Choose \(\gamma(\beta)\le \beta\) with \(\abs{\gamma(\beta)}=j-s=\min\{j,k\}\) and set \(\eta(\beta)=\beta-\gamma(\beta)\), whence \(\abs{\eta(\beta)}=s\). Then
\[
\norm{\delta^{\beta}P_{t}(a)}_{L^{p}(\Tq)}
\le
\frac{\widetilde C_{r,n}}{t^{s/2}}\norm{\delta^{\gamma(\beta)}a}_{L^{p}(\Tq)}
\le
\widetilde C_{r,n}\bigl(1+t^{-r/2}\bigr)\norm{\delta^{\gamma(\beta)}a}_{L^{p}(\Tq)}.
\]
Summing over all \(\beta\) with \(\abs{\beta}\le k+r\) and using the finiteness of the index sets yields
\[
\norm{P_{t}a}_{W^{k+r,p}_{\theta}}
\le
C_{k,r,n}\bigl(1+t^{-r/2}\bigr)\norm{a}_{W^{k,p}_{\theta}}.
\]
Let \(a\in W^{k,p}_{\theta}\), and choose \(a_{j}\in\A\) such that \(a_{j}\to a\) in \(W^{k,p}_{\theta}\). The estimate above gives
\[
\|P_{t}a_{j}-P_{t}a_{i}\|_{W^{k+r,p}_{\theta}}
\le
C_{k,r,n}\bigl(1+t^{-r/2}\bigr)
\|a_{j}-a_{i}\|_{W^{k,p}_{\theta}}.
\]
Thus \(P_{t}a_{j}\) is Cauchy in \(W^{k+r,p}_{\theta}\). Its limit is independent of the approximating sequence, and this defines the bounded extension of \(P_{t}\) from \(W^{k,p}_{\theta}\) into \(W^{k+r,p}_{\theta}\).
\end{proof}

\begin{corollary}\label{Cor3.8}
Let \(1<p<\infty\). There exists a constant \(C_{n}>0\), independent of \(\theta\) and \(p\), such that for every \(a\in L^{p}(\Tq)\) and every \(t>0\),
\[
\norm{\Hesst(P_{t}a)}_{\ell^{2}_{n}(L^{p}(\Tq))}
\le
\frac{C_{n}}{t}\norm{a}_{L^{p}(\Tq)}.
\]
\end{corollary}

\begin{proof}
Using the bounded extensions of the operators \(\delta^{\alpha}P_{t}\) for \(\abs{\alpha}=2\), we may write
\[
\begin{aligned}
\norm{\Hesst(P_{t}a)}_{\ell^{2}_{n}(L^{p}(\Tq))}^{2}
&=
\sum_{i=1}^{n}\norm{\delta_{i}^{2}P_{t}(a)}_{L^{p}(\Tq)}^{2}
+
\sum_{\substack{1\le i,j\le n\\ i\ne j}}\norm{\delta_{i}\delta_{j}P_{t}(a)}_{L^{p}(\Tq)}^{2} \\
&=
\sum_{i=1}^{n}\norm{\delta_{i}^{2}P_{t}(a)}_{L^{p}(\Tq)}^{2}
+
2\sum_{1\le i<j\le n}\norm{\delta_{i}\delta_{j}P_{t}(a)}_{L^{p}(\Tq)}^{2} \\
&\le
2\sum_{\alpha\in\Ical_{2}}\norm{\delta^{\alpha}P_{t}(a)}_{L^{p}(\Tq)}^{2}\\
&\le
\frac{C_{n}^{2}}{t^{2}}\norm{a}_{L^{p}(\Tq)}^{2},
\end{aligned}
\]
after enlarging the constant if necessary. Taking square roots completes the proof.
\end{proof}

\begin{remark}
Taking \(\alpha=0\) in Theorem~\ref{Thm1.1}(i) gives
\[
\norm{L^{\ell}P_{t}(a)}_{L^{p}(\Tq)}
\le
\frac{C_{\ell,n}}{t^{\ell}}\norm{a}_{L^{p}(\Tq)},
\qquad 1<p<\infty,
\]
which shows that the heat semigroup gains one power of \(t^{-1}\) for each application of the generator.
\end{remark}

\begin{remark}
The pointwise estimates proved in this section complement the integrated semigroup characterizations of function spaces on quantum tori obtained in \cite{XiongXuYin}. In particular, the Theorem~\ref{Thm1.1}(i) isolates the single-time operator norm of \(\delta^{\alpha}L^{\ell}P_{t}\), rather than a Besov- or Triebel-Lizorkin-type square function built from the whole semigroup orbit.
\end{remark}

\section{Sobolev algebra and polynomial nonlinear estimates}\label{sec:sobolev-algebra}

We prove the nonlinear estimates needed for the semilinear heat equation. The arguments are included in detail because the noncommutative product is not pointwise multiplication. The key observation is that the Fourier coefficients of a product are given by a twisted convolution whose twisting factor has modulus one.

For \(r,s\in\Z^{n}\), there exists a scalar \(\omega_{\theta}(r,s)\in\C\), \(\abs{\omega_{\theta}(r,s)}=1\), such that
\[
U^{r}U^{s}=\omega_{\theta}(r,s)U^{r+s}.
\]
Consequently, if \(a,b\in\A\), then
\begin{equation}\label{eq4.1}
\widehat{ab}(m)
=
\sum_{r+s=m}\omega_{\theta}(r,s)\widehat a(r)\widehat b(s),
\qquad m\in\Z^{n}.
\end{equation}
Only the bound \(\abs{\omega_{\theta}(r,s)}=1\) will be used.

\begin{lemma}\label{Lem4.1}
Let \(k\in\N\) satisfy \(k>n/2\). Then there exists a constant \(C_{k,n}>0\), independent of \(\theta\), such that
\[
\sum_{m\in\Z^{n}}\abs{\widehat a(m)}
\le
C_{k,n}\norm{a}_{H^{k}_{\theta}},
\qquad a\in H^{k}_{\theta}.
\]
In particular,
\[
\norm{a}_{L^{\infty}(\Tq)}
\le
C_{k,n}\norm{a}_{H^{k}_{\theta}},
\qquad a\in H^{k}_{\theta}.
\]
\end{lemma}

\begin{proof}
For \(a\in\A\), the Cauchy--Schwarz inequality gives
\[
\sum_{m\in\Z^{n}}\abs{\widehat a(m)}
\le
\left(\sum_{m\in\Z^{n}}(1+4\pi^{2}\abs{m}^{2})^{-k}\right)^{1/2}
\left(\sum_{m\in\Z^{n}}(1+4\pi^{2}\abs{m}^{2})^{k}\abs{\widehat a(m)}^{2}\right)^{1/2}.
\]
The first series is finite precisely under the condition \(k>n/2\). This proves the \(\ell^{1}\)-estimate. Since each \(U^{m}\) is unitary,
\[
\norm{a}_{L^{\infty}(\Tq)}
\le
\sum_{m\in\Z^{n}}\abs{\widehat a(m)}.
\]
The conclusion for all \(a\in H^{k}_{\theta}\) follows by density of \(\A\) in \(H^{k}_{\theta}\).
\end{proof}

\begin{proposition}\label{Prop4.2}
Let \(k\in\N\) satisfy \(k>n/2\). Then \(H^{k}_{\theta}\) is a Banach algebra. More precisely, there exists a constant \(A_{k,n}>0\), independent of \(\theta\), such that
\[
\norm{ab}_{H^{k}_{\theta}}
\le
A_{k,n}\norm{a}_{H^{k}_{\theta}}\norm{b}_{H^{k}_{\theta}},
\qquad a,b\in H^{k}_{\theta}.
\]
\end{proposition}

\begin{proof}
We first prove the estimate for \(a,b\in\A\). Set
\[
\langle m\rangle=(1+4\pi^{2}\abs{m}^{2})^{1/2}.
\]
For \(m=r+s\), the elementary inequality
\[
\langle m\rangle^{k}
\le
C_{k}\bigl(\langle r\rangle^{k}+\langle s\rangle^{k}\bigr)
\]
holds. Using \eqref{eq4.1} and \(\abs{\omega_{\theta}(r,s)}=1\), we obtain
\[
\langle m\rangle^{k}\abs{\widehat{ab}(m)}
\le
C_{k}\sum_{r+s=m}\langle r\rangle^{k}\abs{\widehat a(r)}\abs{\widehat b(s)}
+
C_{k}\sum_{r+s=m}\abs{\widehat a(r)}\langle s\rangle^{k}\abs{\widehat b(s)}.
\]
Taking the \(\ell^{2}(\Z^{n})\)-norm and using Young's convolution inequality \(\ell^{2}*\ell^{1}\to\ell^{2}\), we get
\[
\norm{ab}_{H^{k}_{\theta}}
\le
C_{k}\norm{a}_{H^{k}_{\theta}}\sum_{s\in\Z^{n}}\abs{\widehat b(s)}
+
C_{k}\left(\sum_{r\in\Z^{n}}\abs{\widehat a(r)}\right)\norm{b}_{H^{k}_{\theta}}.
\]
The Lemma~\ref{Lem4.1} gives
\[
\norm{ab}_{H^{k}_{\theta}}
\le
A_{k,n}\norm{a}_{H^{k}_{\theta}}\norm{b}_{H^{k}_{\theta}}.
\]
The extension to arbitrary \(a,b\in H^{k}_{\theta}\) follows by approximating them in \(H^{k}_{\theta}\) by elements of \(\A\) and using the estimate above.
\end{proof}

\begin{proposition}\label{Prop4.3}
Let \(k\in\N\) satisfy \(k>n/2\). Let \(q\in\N\), and for \(0\le\nu\le q\) and \(1\le\mu\le N_{\nu}\), let
\[
b_{\nu,\mu,0},\dots,b_{\nu,\mu,\nu}\in H^{k}_{\theta}.
\]
Define the noncommutative polynomial
\begin{equation}\label{eq4.2}
\mathcal P(u)
=
\sum_{\nu=0}^{q}\sum_{\mu=1}^{N_{\nu}}
 b_{\nu,\mu,0}u b_{\nu,\mu,1}u\cdots u b_{\nu,\mu,\nu},
\end{equation}
where the term with \(\nu=0\) is interpreted as \(b_{0,\mu,0}\). Then \(\mathcal P:H^{k}_{\theta}\to H^{k}_{\theta}\) is locally Lipschitz. More precisely, there exists a constant \(C_{\mathcal P,k,n}>0\) such that
\begin{equation}\label{eq4.3}
\norm{\mathcal P(u)}_{H^{k}_{\theta}}
\le
C_{\mathcal P,k,n}\bigl(1+\norm{u}_{H^{k}_{\theta}}^{q}\bigr),
\qquad u\in H^{k}_{\theta},
\end{equation}
and for every \(R>0\) there exists \(C_{\mathcal P,k,n,R}>0\) such that
\begin{equation}\label{eq4.4}
\norm{\mathcal P(u)-\mathcal P(v)}_{H^{k}_{\theta}}
\le
C_{\mathcal P,k,n,R}\norm{u-v}_{H^{k}_{\theta}}
\end{equation}
whenever \(\norm{u}_{H^{k}_{\theta}}\le R\) and \(\norm{v}_{H^{k}_{\theta}}\le R\).
\end{proposition}

\begin{proof}
The growth estimate follows by repeated use of Proposition~\ref{Prop4.2}. Indeed, for each monomial of degree \(\nu\),
\[
\norm{b_{\nu,\mu,0}u b_{\nu,\mu,1}u\cdots u b_{\nu,\mu,\nu}}_{H^{k}_{\theta}}
\le
A_{k,n}^{2\nu}
\left(\prod_{j=0}^{\nu}\norm{b_{\nu,\mu,j}}_{H^{k}_{\theta}}\right)
\norm{u}_{H^{k}_{\theta}}^{\nu}.
\]
Summing over \(\nu\) and \(\mu\) gives \eqref{eq4.3}.

For the Lipschitz estimate, fix a monomial
\[
M(u)=b_{0}u b_{1}u\cdots u b_{\nu}.
\]
If \(\nu=0\), then \(M(u)-M(v)=0\). If \(\nu\ge1\), the telescopic identity gives
\[
M(u)-M(v)
=
\sum_{j=1}^{\nu}
 b_{0}v b_{1}\cdots v b_{j-1}(u-v)b_{j}u b_{j+1}\cdots u b_{\nu}.
\]
Using Proposition~\ref{Prop4.2} on each term and assuming \(\norm{u}_{H^{k}_{\theta}},\norm{v}_{H^{k}_{\theta}}\le R\), we obtain
\[
\norm{M(u)-M(v)}_{H^{k}_{\theta}}
\le
C_{M,k,n,R}\norm{u-v}_{H^{k}_{\theta}}.
\]
Summing over all monomials proves \eqref{eq4.4}.
\end{proof}

\begin{remark}\label{Rem4.4}
The proof is given in the Hilbert scale because it only uses the twisted Fourier convolution and elementary \(\ell^{1}\)-\(\ell^{2}\) estimates. More general algebra and Moser-type estimates in \(W^{k,p}_{\theta}\), \(1<p<\infty\), can be obtained from the Sobolev and Besov theory on quantum tori. The Hilbert version is sufficient for the polynomial semilinear heat equations considered below and keeps the nonlinear part independent of additional paraproduct machinery.
\end{remark}

\section{Semilinear heat equations}\label{sec:semilinear-heat}

Let \(k\in\N\) satisfy \(k>n/2\). We consider
\begin{equation}\label{eq5.1}
\begin{cases}
\partial_t u+Lu=\mathcal N(u), & t>0,\\
u(0)=u_0,
\end{cases}
\end{equation}
where \(u_0\in H^{k}_{\theta}\) and \(\mathcal N\) is a nonlinear map on \(H^{k}_{\theta}\).

\begin{definition}\label{Def5.1}
Let \(T>0\). A function \(u\in C([0,T];H^{k}_{\theta})\) is called a mild solution of \eqref{eq5.1} on \([0,T]\) if
\[
u(t)=P_tu_0+\int_0^t P_{t-s}\mathcal N(u(s))\,ds,
\qquad 0\le t\le T,
\]
where the integral is understood as a Bochner integral in \(H^{k}_{\theta}\).
\end{definition}

\begin{proof}[Proof of Theorem~\ref{Thm1.2}]
Put \(X=H^{k}_{\theta}\). The heat semigroup is a contraction on \(X\), because
\[
\norm{P_t a}_{H^{k}_{\theta}}^{2}
=
\sum_{m\in\Z^{n}}(1+4\pi^{2}\abs{m}^{2})^{k}e^{-8\pi^{2}\abs{m}^{2}t}\abs{\widehat a(m)}^{2}
\le
\norm{a}_{H^{k}_{\theta}}^{2}.
\]
The strong continuity of \(P_t\) on \(X\) follows from dominated convergence.
Indeed, for \(a\in H^k_\theta\),
\[
\|P_ta-a\|_{H^k_\theta}^2
=
\sum_{m\in\Z^n}
(1+4\pi^2\abs m^2)^k
\abs{e^{-4\pi^2\abs m^2 t}-1}^2
\abs{\widehat a(m)}^2 .
\]
For each fixed \(m\in\Z^n\), the factor
\(\abs{e^{-4\pi^2\abs m^2 t}-1}^2\) tends to \(0\) as \(t\downarrow0\), and it is bounded by \(4\).
Since
\[
\sum_{m\in\Z^n}
(1+4\pi^2\abs m^2)^k\abs{\widehat a(m)}^2<\infty,
\]
the dominated convergence theorem gives
\[
\lim_{t\downarrow0}\|P_ta-a\|_{H^k_\theta}=0.
\]

Fix \(u_0\in X\). Choose \(R>2\norm{u_0}_{X}\). Since \(\mathcal N\) is locally Lipschitz, there exist constants \(L_R,M_R>0\) such that
\[
\norm{\mathcal N(v)-\mathcal N(w)}_{X}
\le
L_R\norm{v-w}_{X},
\qquad
\norm{\mathcal N(v)}_{X}\le M_R
\]
whenever \(\norm{v}_{X},\norm{w}_{X}\le R\). Choose \(T>0\) so small that
\[
\norm{u_0}_{X}+TM_R\le R,
\qquad
TL_R\le \frac{1}{2}.
\]
Let
\[
\mathcal B_R=\left\{u\in C([0,T];X):\sup_{0\le t\le T}\norm{u(t)}_{X}\le R\right\}.
\]
For \(u\in\mathcal B_R\), define
\[
(\Phi u)(t)=P_tu_0+\int_0^tP_{t-s}\mathcal N(u(s))\,ds.
\]
Clearly, \(\Phi\mathcal B_R\subset\mathcal B_R\) since 
$$\|\Phi u\|_{C([0,T];X)}\leq\norm{u_0}_{X}+TM_R\le R,$$
and for \(u,v\in\mathcal B_R\),
\[
\norm{\Phi u-\Phi v}_{C([0,T];X)}
\le
TL_R\norm{u-v}_{C([0,T];X)}
\le
\frac{1}{2}\norm{u-v}_{C([0,T];X)}.
\]
The contraction mapping theorem gives a unique mild solution on \([0,T]\). The same estimate on shorter intervals gives local uniqueness. This completes (i).

We define
\[
T_{\max}
=
\sup\left\{
T>0:\text{there exists a mild solution in }C([0,T];X)
\right\}.
\]
By the local existence result, \(T_{\max}>0\). The local uniqueness just proved
allows us to paste together the local solutions, and hence there is a unique
mild solution on every interval \([0,T]\) with \(0<T<T_{\max}\).

We now prove the blow-up alternative. Suppose that \(T_{\max}<\infty\) and
\[
\sup_{0\le t<T_{\max}}\|u(t)\|_X<\infty.
\]
Then there exists \(R>0\) such that \(\|u(t)\|_X\le R\) for all
\(0\le t<T_{\max}\). Since the local existence time obtained above depends only
on the size of the initial datum in \(X\), there exists \(\tau>0\), depending
only on \(R\), such that for every \(t_0<T_{\max}\) the problem with initial
datum \(u(t_0)\) has a mild solution on \([t_0,t_0+\tau]\). Choosing \(t_0\)
sufficiently close to \(T_{\max}\), we have \(t_0+\tau>T_{\max}\), which extends
the solution beyond \(T_{\max}\). This contradicts the definition of
\(T_{\max}\). Therefore
\[
\lim_{t\to T_{\max}}\|u(t)\|_X=\infty.
\]
This completes (ii).

For continuous dependence, let \(u\) and \(v\) be two mild solutions on \([0,T]\) such that
\[
\|u(t)\|_X,\|v(t)\|_X\le R,
\qquad 0\le t\le T.
\]
If \(L_R\) is a Lipschitz constant for \(\mathcal N\) on the ball of radius \(R\), then the mild formula and the contraction property of \(P_t\) on \(X\) give
\[
\|u(t)-v(t)\|_X
\le
\|u(0)-v(0)\|_X+L_R\int_0^t\|u(s)-v(s)\|_X\,ds.
\]
Gronwall's inequality yields
\[
\sup_{0\le t\le T}\|u(t)-v(t)\|_X
\le e^{L_RT}\|u(0)-v(0)\|_X.
\]
This proves local Lipschitz continuous dependence on the initial datum. This completes (iii).

We prove the smoothing estimate. Let \(0<T<T_{\max}\). From the mild formula and Corollary~\ref{Cor3.7} with \(p=2\),  for \(0<t\le T\),
\[
\begin{aligned}
\norm{u(t)}_{H^{k+1}_{\theta}}
&\le
C_{k,n}\bigl(1+t^{-1/2}\bigr)\norm{u_0}_{H^{k}_{\theta}} \\
&\quad +C_{k,n}\int_0^t\bigl(1+(t-s)^{-1/2}\bigr)\norm{\mathcal N(u(s))}_{H^{k}_{\theta}}\,ds \\
&\le
C_{k,n}\bigl(1+t^{-1/2}\bigr)\norm{u_0}_{H^{k}_{\theta}}
+C_{k,n}M_T\bigl(t+t^{1/2}\bigr).
\end{aligned}
\]
This proves \eqref{eq1.1} and, in particular, \(u(t)\in H^{k+1}_{\theta}\) for every \(t>0\).

It remains only to justify the continuity with values in \(H^{k+1}_{\theta}\) away from \(t=0\). Fix \(0<\varepsilon<T\) and put \(\sigma=\varepsilon/2\). Since \(u(\sigma)\in H^{k+1}_{\theta}\), the mild formula may be restarted at time \(\sigma\):
\[
u(t)=P_{t-\sigma}u(\sigma)+\int_{\sigma}^{t}P_{t-s}\mathcal N(u(s))\,ds,
\qquad \varepsilon\le t\le T.
\]
The first term is continuous in \(H^{k+1}_{\theta}\) by strong continuity of \(P_t\) on that space. For the second one, calling it the Duhamel term, set \(g(s)=\mathcal N(u(s))\). Clearly, \(g\in C([\sigma,T];H^k_\theta)\), and the Corollary~\ref{Cor3.7} gives
\[
\|P_{t-s}g(s)\|_{H^{k+1}_\theta}
\le C_{k,n}\bigl(1+(t-s)^{-1/2}\bigr)\sup_{\sigma\le r\le T}\|g(r)\|_{H^k_\theta}.
\]
The right-hand side is integrable in \(s\) near \(s=t\). Splitting the integral into \([\sigma,t-\eta]\) and \([t-\eta,t]\), using strong continuity on the first part and then letting \(\eta\downarrow0\) on the second part, proves the continuity of the Duhamel term in \(H^{k+1}_{\theta}\). Hence \(u\in C([\varepsilon,T];H^{k+1}_{\theta})\), and since \(\varepsilon>0\) is arbitrary, \(u\in C((0,T];H^{k+1}_{\theta})\). This completes (iv).

Finally, if \(u_0\in H^{k+r}_{\theta}\) and \(\mathcal N\) is locally Lipschitz on \(H^{k+r}_{\theta}\), the fixed-point argument can be repeated with \(X=H^{k+r}_{\theta}\). Uniqueness in \(H^{k}_{\theta}\) identifies the two solutions. This proves the persistence of higher regularity in (v) and the proof is complete.
\end{proof}

\begin{corollary}\label{Cor5.2}
Let \(k\in\N\) satisfy \(k>n/2\), and let \(\mathcal P\) be the noncommutative polynomial defined in \eqref{eq4.2} with coefficients in \(H^{k}_{\theta}\). Then, for every \(u_0\in H^{k}_{\theta}\), the polynomial semilinear heat equation
\[
\partial_t u+Lu=\mathcal P(u),
\qquad
u(0)=u_0,
\]
has a unique local mild solution in \(C([0,T];H^{k}_{\theta})\) (where "local" corresponds to "local time"). It satisfies the blow-up alternative, continuous dependence on the initial datum, and the smoothing estimate in Theorem~\ref{Thm1.2}.
\end{corollary}

\begin{proof}
By Proposition~\ref{Prop4.3}, the map \(\mathcal P:H^{k}_{\theta}\to H^{k}_{\theta}\) is locally Lipschitz. The conclusion follows from Theorem~\ref{Thm1.2}.
\end{proof}

\begin{corollary}\label{Cor5.3}
Let \(k\in\N\) satisfy \(k>n/2\). Assume that all coefficients of \(\mathcal P\) in \eqref{eq4.2} belong to \(H^{k+r}_{\theta}\) for every \(r\in\N\). If \(u_0\in H^{k}_{\theta}\), then the corresponding mild solution satisfies
\[
u\in C((0,T];H^{k+r}_{\theta})
\]
for every \(r\in\N\) and every \(0<T<T_{\max}\).
\end{corollary}

\begin{proof}
The Theorem~\ref{Thm1.2} gives \(u\in C((0,T];H^{k+1}_{\theta})\). Fix \(\varepsilon>0\). Taking \(u(\varepsilon)\) as new initial data and using Proposition~\ref{Prop4.3} in the space \(H^{k+1}_{\theta}\), we obtain \(u\in C([\varepsilon,T];H^{k+1}_{\theta})\) and then the one-derivative smoothing estimate gives \(u\in C((\varepsilon,T];H^{k+2}_{\theta})\). Iterating this argument proves the claim.
\end{proof}

\begin{remark}\label{Rem5.4}
The fixed-point part of the Theorem~\ref{Thm1.2} only uses the boundedness and strong continuity of \(P_t\) on \(H^{k}_{\theta}\). The higher-order heat estimates enter through the smoothing bound \eqref{eq1.1}. Thus the estimates of Section~\ref{sec:main-estimates} provide quantitative parabolic regularization for nonlinear equations on the noncommutative torus.
\end{remark}

\section*{Acknowledgments}
The authors would like to thank Professor Quanhua Xu and Professor Xiao Xiong from the Institute for Mathematical Sciences, Harbin Institute of Technology, for their kind hospitality during the visit to Harbin Institute of Technology. 

\medskip
{\bf Funding:} This work is supported by the National Natural Science Foundation of China\linebreak (12301145, 12261107, 12561020) and by Yunnan Fundamental Research Projects\linebreak (202401AU070123).

\medskip
{\bf Data availability:}  Data sharing is not applicable to this article as no new data were created or analyzed in this study.

\medskip
{\bf Conflict of interest:} The authors declare that there is no conflict of interest.

\end{document}